\input ./style/arxiv-vmsta.cfg
\documentclass[numbers,compress,v1.0.1]{vmsta}

\volume{4}
\issue{3}
\pubyear{2017}
\firstpage{219}
\lastpage{231}
\doi{10.15559/17-VMSTA84}

\startlocaldefs
\newcommand{\rrVert}{\Vert}
\newcommand{\llVert}{\Vert}
\newcommand{\rrvert}{\vert}
\newcommand{\llvert}{\vert}
\allowdisplaybreaks

\endlocaldefs

\newtheorem{thm}{Theorem}

\newtheorem{note}{Note}

\theoremstyle{definition}
\newtheorem{example}{Example}

\begin{document}
\begin{frontmatter}

\title{Quantifying non-monotonicity of functions and the~lack of positivity in signed measures}

\author[a]{\inits{Yu.}\fnm{Youri}\snm{Davydov}\corref{cor1}}\email{davydov.youri@gmail.com}
\cortext[cor1]{Corresponding author.}
\author[b]{\inits{R.}\fnm{Ri\v{c}ardas}\snm{Zitikis}}\email{zitikis@stats.uwo.ca}

\address[a]{Chebyshev Laboratory, St.\,Petersburg State University, Vasilyevsky Island, St.\,Petersburg 199178, Russia}
\address[b]{School of Mathematical and Statistical Sciences, Western University, London, Ontario~N6A~5B7, Canada}

\markboth{Yu. Davydov, R. Zitikis}{Quantifying non-monotonicity of functions}

\begin{abstract}
In various research areas related to decision making, problems and their
solutions frequently rely on certain functions being monotonic. In the case
of non-monotonic functions, one would then wish to quantify their lack of
monotonicity. In this paper we develop a method designed specifically for
this task, including quantification of the lack of positivity, negativity,
or sign-constancy in signed measures. We note relevant applications in
Insurance, Finance, and Economics, and discuss some of them in detail.
\end{abstract}




\received{16 June 2017}
\revised{5 September 2017}
\accepted{5 September 2017}
\publishedonline{28 September 2017}
\end{frontmatter}

\section{Introduction}
In various research areas such as Economics, Finance, Insurance, and
Statistics, problems and their solutions frequently rely on certain
functions being monotonic, as well as on determining the degree of
their monotonicity, or lack of it. For example, the notion of profit\vadjust{\eject}
seekers in Behavioural Economics and Finance is based on increasing
utility functions, which can have varying shapes and thus characterize
subclasses of profit seekers (e.g., \cite{GilMar2009}). In
Reliability Engineering and Risk Assessment (e.g., \cite{LiLi2013}), a
number of notions such as hazard-rate and likelihood-ratio orderings
rely on monotonicity of the ratios of certain functions. The presence
of insurance deductibles and policy limits often change the pattern of
monotonicity of insurance losses (e.g., \cite{Braetal2015}). In
the literature on Statistical Inference, the so-called
monotone-likelihood-ratio family plays an important role.

Due to these and a myriad of other reasons, researchers quite often
restrict themselves to function classes with pre-specified monotonicity
properties. But one may not be comfortable with this element of
subjectivity and would therefore prefer to rely on data-implied shapes
when making decisions. To illustrate the point, we recall, for example,
the work of Bebbington et al. \cite{Bebetal2011} who specifically set out to
determine whether mortality continues to increase, or starts to
decelerate, after a certain species-related late-life age. This is
known in the gerontological literature as the late-life mortality
deceleration phenomenon. Naturally, we refrain from elaborating on this
complex topic and refer for details and further references to the
aforementioned paper.

Monotonicity may indeed be necessary for certain results or properties
to hold, but there are also many instances when monotonicity is just a
sufficient condition. In such cases, a natural question arises: can we
depart from monotonicity and still have valid results? Furthermore, in
some cases, monotonicity may not even be expected to hold, though
perhaps be desirable, and so developing techniques for quantifying the
lack of monotonicity becomes of interest. Several results in this
direction have recently been proposed in the literature (e.g.,
\cite{DavZit2005,YitSch2013,QoyZit2014,QoyZit2015}; and references therein). All in all, these are some of the
problems, solutions to which ask for indices that could be used to assess
monotonicity, or lack of it. In the following sections we introduce and
discuss several such indices, each designed to reveal different
monotonicity aspects.\looseness=1

We have organized the rest of the paper as follows. In Section~\ref{motive} we present a specific example, driven by insurance and
econometric considerations, which not only motivates present research
but also highlights the main idea adopted in this paper. In Section~\ref{section-2}
we introduce and discuss indices of lack of increase, decrease, and
monotonicity. We also suggest a convenient numerical procedure for a
quick calculation of the indices at any desirable precision. In
Section~\ref{sec4} we introduce orderings of functions according to the values
of their indices and argue in favour of using normalized versions of
the indices. In Section~\ref{section-strict} we introduce a stricter
notion of ordering, and in Section~\ref{section-hj} we develop a theory
for quantifying the lack of positivity (or negativity) in signed
measures. Section~\ref{section-conclude} concludes the paper.

\section{Motivating example}
\label{motive}

Insurance losses are non-negative random variables $X\ge0$, and the
expected value $\mathbf{E}[X]$ is called the net premium. All
practically-sound premium calculation principles (pcp's), which are
functionals $\pi$ assigning non-negative finite or infinite values to
$X$'s, are such that $\pi[X]\ge\mathbf{E}[X]$. The latter property is
called non-negative loading of $\pi$.\vadjust{\eject}

As an example, consider the following `dual' version of the weighted
pcp (\cite{FurZit2008,FurZit2009}; and references therein):
%
\begin{equation}
\label{0pi_w} \pi_w[X]=\frac{\textbf{E}[F^{-1}(U) w(U)]}{\textbf{E}[w(U)]},
\end{equation}
where $U$ is the random variable uniform on $[0,1]$, $F^{-1}$ is the
quantile function of $X$ defined by the equation $F^{-1}(p)=\inf\{
x~:~ F(x)\ge p\}$, and $w:[0,1]\to[0,\infty]$ is
an appropriately chosen weight function, whose illustrative examples
will be provided in a moment. We of course assume that the two
expectations in the definition of $\pi_w$ are well-defined and finite,
and the denominator is not zero.

When dealing with insurance losses, researchers usually choose
non-decrea\-sing weight functions in order to have non-negative loading
of $\pi_w$. For example, $w(t)=\mathbf{1}\{t>p \}$ for any parameter
$p\in(0,1)$ is non-decreasing, and it turns $\pi_w$ into the average
value at risk, also known as tail conditional expectation. Another
example is $w(t)=\nu(1-t)^{\nu-1}$ with parameter $\nu>0$. If $\nu\in
(0,1]$, then $w$ is non-decreasing, and $\pi_w$ becomes the
proportional hazards pcp \cite{Wan1995,Wan1996}. If $\nu\ge1$, then $w$
is non-increasing, and $\pi_w$ becomes the (absolute) $S$-Gini index of
economic equality (e.g., \cite{ZitGas2002}; and references therein).

Note that the pcp $\pi_w$ can be rewritten as the weighted integral
%
\begin{equation}
\label{0pi_w-small} \pi_w=\int_0^1
F^{-1}(t)w^*(t)\,\xch{\mathrm{d}t,}{\mathrm{d}t}
\end{equation}
with the normalized weight function
$w^*(t)=w(t)/ \int_0^1 w(u)\mathrm{d}u$, which is a probability density
function, because $w(t)\ge0$ for all $t\in[0,1]$ and $\int_0^1
w(u)\mathrm{d}u\in(0,\infty)$. This representation of $\pi_w$ connects
our present research with the dual utility theory (\cite{Yaa1987,Qui1993}; and references therein) that has arisen as a prominent
counterpart to the classical utility theory of von Neuman and
Morgenstern \cite{NeuMor1944}. It is also important to mention that in Insurance,
integral (\ref{0pi_w-small}) plays a very prominent role and is known
as the distortion risk measure (e.g., \cite{Denetal2005}; and
references therein).

In general, the function $w$ and its normalized version $w^*$ may not
be monotonic, as it depends on the shape of the probability density
function of the random variable $W \in[0,1]$ in the following
reformulation of $\pi_w$:
%
\begin{equation}
\label{0pi_w-5} \pi_w[X]=\textbf{E} \xch{\bigl[F^{-1}(W) \bigr].}{\bigl[F^{-1}(W) \bigr]}
\end{equation}
In the econometric language, $\textbf{E}[F^{-1}(W)]$ means the average
income (assuming that $X$ stands for `income') possessed by individuals
whose positions on the society's income-percentile scale are modelled
by $W$, which is, naturally, a random variable.

Hence, we are interested when $\textbf{E}[F^{-1}(W)]$ is at least
$\textbf{E}[W]$ (insurance perspective) or at most $\textbf{E}[W]$
(income inequality perspective). In view of equations (\ref{0pi_w}) and
(\ref{0pi_w-small}), our task reduces to verifying whether or not
%
\begin{equation}
\mathbf{cov} \bigl[F^{-1}(U),w(U) \bigr]\ge0. \label{pcp-2}
\end{equation}
If the function $w$ happens to be non-decreasing (as in insurance),
then bound (\ref{pcp-2}) holds (e.g., \cite{Leh1966}). For example, with
parameter $\lambda>0$, the weight function $ w(s)=s^{\lambda} $ leads
to the size-biased pcp,
$w(s)=e^{\lambda s }$ to the Esscher pcp,
$w(s)=1-e^{-\lambda s }$ to the Kamps pcp, and the already noted function
$w(s)=\mathbf{1}\{s>\lambda\}$ leads to the average value at risk.

As already noted, broader contexts than that of classical insurance
suggest various shapes of probability distortions and thus lead to
functions $w$ that are not necessarily monotonic. Indeed, in view of
representation (\ref{0pi_w-5}), the average income of individuals
depends on the distribution of $W$, whose shape is governed by societal
opportunities that the individuals are exposed to. A natural question arises:
\[
\textrm{What shapes of $w$ could ensure property (\ref{pcp-2})?}
\]
To answer this question in an illuminating way, let $w$ be absolutely
continuous and such that $w(0)=0$, which are sound assumptions from the
practical point of view. Denote the density of $w$ by $w'$, and let
$\mathbf{1}_{\{u> t\}}$ be the indicator. We have
\begin{align*}
\mathbf{cov} \bigl[F^{-1}(U),w(U) \bigr] &=\mathbf{cov}
\Biggl[F^{-1}(U),\int_0^{U}w'(t)\,\mathrm{d}t \Biggr]
\\
&=\int_0^{1}\mathbf{cov} \bigl[F^{-1}(U),
\mathbf{1}_{\{U> t\}} \bigr] w'(t)\,\mathrm{d}t.
\end{align*}
The function $v(t)=\mathbf{cov}[F^{-1}(U),\mathbf{1}_{\{U> t\}}]$ is
non-negative (e.g., \cite{Leh1966}), and so is also the integral $\theta
=\int_0^{1}v(t)dt$, which makes $v(t)/\theta$ a probability density
function. Denote the corresponding distribution function by $V$. We
have the equations
\begin{align*}
\mathbf{cov} \bigl[F^{-1}(U),w(U) \bigr] &=\theta\int
_0^{1}w'\,\mathrm{d}V
\\
&= \theta\int_0^{1}w'\circ
V^{-1}\,\mathrm{d}\lambda, \label{pcp-3a}
\end{align*}
where $\lambda$ denotes the Lebesgue measure. Hence, property (\ref
{pcp-2}) is equivalent to the bound
$\int_0^{1}g\mathrm{d}\lambda\ge0$ with $g=w'\circ V^{-1}$, and it
can of course be rewritten as the inequality
%
\begin{equation}
\int_0^{1}g^{+}\,\mathrm{d}\lambda\ge
\int_0^{1 }g^{-}\,\mathrm{d}\lambda,
\label{pcp-3bi}
\end{equation}
where $g^{+}=\max\{g , 0\} $ and $g^{-}=\max\{ -g, 0\}$. Bound (\ref
{pcp-3bi}) says that, in average with respect to the Lebesgue measure,
the positive part $g^{+}$ must be larger than the negative part $g^{-}$.

Those familiar with asset pricing will immediately see how inequality
(\ref{pcp-3bi}), especially when reformulated as
%
\begin{equation}
{\int_0^{1}g^{+}\,\mathrm{d}\lambda\over\int_0^{1 }g^{-}\mathrm
{d}\lambda} \ge \xch{1}{1,} \label{pcp-3b}
\end{equation}
is connected to the gain--loss ratio \cite{BerLed2000}, as
well as to the Omega ratio \cite{KeaSha2002}. Furthermore,
we can reformulate bound (\ref{pcp-3b}) as
%
\begin{equation}
{\int_0^{1}g^{+}\,\mathrm{d}\lambda\over\int_0^{1 }\llvert g\rrvert \,\mathrm{d}\lambda
} \ge{1\over2}, \label{pcp-3b-p}
\end{equation}
with the ratio on the left-hand side being equal to $0$ when the
function $g$ is non-positive, and equal to $1$ when $g$ is
non-negative. Hence, bound (\ref{pcp-3b-p}) says that the ratio must be
at least $1/2$, which means that, in average, the function $g$ must be
more positive than negative.

The above interpretations have shaped our considerations in the present paper,
and have led toward the construction of monotonicity indices that we introduce and
discuss next.

\section{Assessing lack of monotonicity in functions}
\label{section-2}

We are interested in assessing monotonicity of a function $g_0$ on an
interval $[a,b]$. Since shifting the function up or down, left or
right, does not distort its monotonicity, we therefore `standardize'
$g_0$ into
%
\begin{equation}
\label{gg} g(x)=g_0(x+a)-g_0(a)
\end{equation}
defined on the interval $[0,y]$ with $y=b-a$. Note that $g$ satisfies
the boundary condition $g(0)=0$. We assume that $g_0$ and thus $g$ are
absolutely continuous.

Let $ \mathcal{F}_y$ denote the set of all absolutely continuous
functions $f$ on the interval $[0,y]$ such that $f(0)=0$. Denote the
total variation of $f\in\mathcal{F}_y$ on the interval $[0,y]$ by
$\Vert f \Vert_y $, that is, $\Vert f \Vert_y =\int_0^y |f'|\mathrm
{d}\lambda$. Furthermore, let $\mathcal{F}^{+}_y $ denote the set of
all $f\in\mathcal{F}_y$ that are non-decreasing. For any $g\in\mathcal
{F}_y$, we define its index of lack of increase (LOI) as the distance
between $g$ and the set $\mathcal{F}^{+}_y $, that is,
%
\begin{equation}
\mathrm{LOI}_y(g) =\inf \Biggl\{ \int_0^y
\bigl\llvert g'-f' \bigr\rrvert \,\mathrm{d}\lambda~:~ f
\in\mathcal {F}^{+}_y \Biggr\} . \label{def-00}
\end{equation}
Obviously, if $g$ is non-decreasing, then $\mathrm{LOI}_y(g)=0$, and
the larger the value of $\mathrm{LOI}_y(g)$, the farther the function
$g$ is from being non-decreasing on the interval $[0,y]$. For the
function $g_0$, which is where $g$ originates from, the LOI of $g_0$ on
the interval $[a,b]$ is
\[
\mathrm{LOI}_{[a,b]}(g_0):=\mathrm{LOI}_y(g).
\]
(Throughout the paper we occasionally use `$:=$,' when a need arises to
emphasize that certain equations are by definition.) Determining
$\mathrm{LOI}_y(g)$ using definition (\ref{def-00}) is not
straightforward. To facilitate the task, we next give an integral
representation of the index.

\begin{thm}\label{theorem21}
The infimum in definition (\ref{def-00}) is attained at a function
$f_1\in\mathcal{F}^{+}_y$ such that $f_1'=(g')^{+}$, and thus
%
\begin{equation}
\label{loi-2} \mathrm{LOI}_y(g)=\int_0^y
\bigl(g' \bigr)^{-}\,\mathrm{d}\lambda.
\end{equation}
When $g$ originates from $g_0$ via equation (\ref{gg}), we have
%
\begin{equation}
\label{loi-2a} \mathrm{LOI}_{[a,b]}(g_0)=\int
_a^b \bigl(g_0'
\bigr)^{-}\,\mathrm{d}\lambda.
\end{equation}
\end{thm}

Theorem~\ref{theorem21}, though easy to prove directly, follows
immediately from a more general result, Theorem~\ref{th-0} of Section~\ref{section-hj}, and we thus do not provide any more details. The
index of lack of decrease (LOD) is defined analogously. Namely, in the
computationally convenient form, it is given by the equation
\[
\mathrm{LOD}_y(g)=\int_0^y
\bigl(g' \bigr)^{+}\,\mathrm{d}\lambda .
\]
When $g$ originates from $g_0$ via equation (\ref{gg}), we have
\[
\mathrm{LOD}_{[a,b]}(g_0)=\int_a^b
\bigl(g_0' \bigr)^{+}\,\mathrm{d}\lambda.
\]
In turn, the index of lack of monotonicity (LOM) of $g$ is given by the equation
%
\begin{equation}
\label{ind-mon} \mathrm{LOM}_y(g) =2 \min \bigl\{
\mathrm{LOI}_y(g), \mathrm{LOD}_y(g) \bigr\},
\end{equation}
and when $g$ originates from $g_0$ via equation (\ref{gg}), we have
%
\begin{equation}
\label{ind-mon-2} \mathrm{LOM}_{[a,b]}(g_0)=2 \min \bigl\{
\mathrm{LOI}_{[a,b]}(g_0), \mathrm{LOD}_{[a,b]}(g_0)
\bigr\} .
\end{equation}
The reason for doubling the minimum on the right-hand sides of
definitions (\ref{ind-mon}) and (\ref{ind-mon-2}) will become clear
from properties below.
\begin{enumerate}
%
\item[A1)] The index $\mathrm{LOI}_y$ is translation
invariant, that is, $\mathrm{LOI}_y(g+\alpha)=\mathrm{LOI}_y(g)$ for
every constant $\alpha\in\mathbf{R}$. The index $\mathrm{LOD}_y$ is
also translation invariant.

\item[A2)]The index $\mathrm{LOI}_y$ is positively
homogeneous, that is,
$\mathrm{LOI}_y(\beta g )=\beta\mathrm{LOI}_y(g)$ for every constant
$\beta\ge0$. The index $\mathrm{LOD}_y$ is also positively homogeneous.
\item[A3)]$\mathrm{LOI}_y(\beta g )=(-\beta) \mathrm
{LOD}_y(g)$ for every negative constant $\beta<0$, and thus in
particular, $\mathrm{LOI}_y(- g )=\mathrm{LOD}_y(g)$.
\item[A4)]\label{prop-a4} $\mathrm{LOI}_y(g )+ \mathrm
{LOD}_y(g)=\Vert g \Vert_y $. Consequently, we have the following two
observations:
\begin{enumerate}
%
\item
$\mathrm{LOI}_y(g )$ and $\mathrm{LOD}_y(g)$ do not exceed $ \Vert g
\Vert_y $, with both indices achieving the upper bound. Namely, $\mathrm
{LOI}_y(g)= \Vert g \Vert_y $ whenever $(g')^{+}\equiv0$, and $\mathrm
{LOD}_y(g)= \Vert g \Vert_y $ whenever $(g')^{-}\equiv0$.
\item
$\min \{ \mathrm{LOI}_y(g), \mathrm{LOD}_y(g) \}\le\Vert g
\Vert_y /2$ and thus $\mathrm{LOM}_y(g) \le\Vert g \Vert_y $, which
justifies the use of the factor $2$ in definitions (\ref{ind-mon}) and
(\ref{ind-mon-2}).
\end{enumerate}
\end{enumerate}

An illustrative example follows.

\begin{example}\label{example-21}\rm
Consider the functions $g_0(z)=\sin(z)$ and $h_0(z)=\cos(z)$ on $[-\pi
/2,\pi]$. Neither of them is monotonic on the interval, but a visual
inspection of their graphs suggests that sine is closer to being
increasing than cosine, which can of course be viewed as a subjective
statement. To substantiate it, we employ the above introduced indices.
First, by lifting and shifting, we turn sine into $g(x)=1-\cos(x)$ and
cosine into $h(x)=\sin(x)$. Since $g'(x)=\sin(x)$ and $h'(x)=\cos(x)$,
we have
%
\begin{gather}
\label{g1} \int_0^{3\pi/2} \bigl(g'\bigr)^{-}\,\mathrm{d}\lambda=1 \quad\textrm{and} \quad \int_0^{3\pi/2} \bigl(g'\bigr)^{+} \,\mathrm{d}\lambda=2,\\
\label{h2} \int_0^{3\pi/2} \bigl(h'\bigr)^{-}\,\mathrm{d}\lambda=2 \quad\textrm{and} \quad \int_0^{3\pi/2} \bigl(h'\bigr)^{+} \,\mathrm{d}\lambda=1.
\end{gather}
Consequently,
\begin{align*}
\mathrm{LOI}_{[-\pi/2,\pi]}(g_0)=\mathrm{LOI}_{3\pi/2}(g)=1,& \quad \mathrm{LOI}_{[-\pi/2,\pi]}(h_0)=\mathrm{LOI}_{3\pi/2}(h)=2,\\
\mathrm{LOD}_{[-\pi/2,\pi]}(g_0)=\mathrm{LOD}_{3\pi/2}(g)=2,& \quad \mathrm{LOD}_{[-\pi/2,\pi]}(h_0)=\mathrm{LOD}_{3\pi/2}(h)=1,\\
\mathrm{LOM}_{[-\pi/2,\pi]}(g_0)=\mathrm{LOM}_{3\pi/2}(g)=2,& \quad \mathrm{LOM}_{[-\pi/2,\pi]}(h_0)=\mathrm{LOM}_{3\pi/2}(h)=2.
\end{align*}
Hence, for example, sine is at the distance $1$ from the set of all
non-decreasing functions on the noted interval, whereas cosine is at
the distance $2$ from the same set. Note also that the total variations
$\int_0^{3\pi/2} |g'|\mathrm{d}\lambda$ and $\int_0^{3\pi/2}
|h'|\mathrm{d}\lambda$ of the two functions are the same, equal to
$3$. This concludes Example~\ref{example-21}.
\end{example}

\begin{note}\label{note-comp}\rm
Example~\ref{example-21} is based on functions for which
the four integrals in equations (\ref{g1}) and (\ref{h2}) are easy to
calculate, but functions arising in applications are frequently quite
unwieldy. For this, we need a numerical procedure for calculating the
integral $\int_0^y H (f')\mathrm{d}\lambda$ for various
transformations $H$, such as $H(x)=x^{-}$ and $H(x)=x^{+}$. A~convenient way is as follows:
\begin{align*}
\int_0^y H \bigl(f' \bigr)\,\mathrm{d}\lambda &\approx\sum_{n=1}^N
{y\over N} H \biggl( {N\over y} \biggl\{ f \biggl(
{n\over N}y \biggr) -f \biggl( {n-1\over N}y \biggr)
\biggr\} \biggr)
\\
&=\sum_{n=1}^N {b-a\over N} H
\biggl({N\over b-a} \biggl\{ f_0 \biggl( a+
{n\over N}(b-a) \biggr)\\
&\quad -f_0 \biggl( a+ {n-1\over N}(b-a) \biggr) \biggr\} \biggr),
\end{align*}
where $f_0$ is the underlying function (on the interval $[a,b]$) and
$f$ is the shifted-and-lifted function (on the interval $[0,y]$)
defined by the equation $f(x)=f_0(x+a)-f_0(a)$ for all $x\in[0,y]$
with $y=b-a$.
\end{note}

The use of the integral $\int_0^y H (f')\mathrm{d}\lambda$ in Note~\ref
{note-comp} hints at the possibility of distances other than $L_1$-norms when defining indices. One may indeed wish to de-emphasize
small values of $f'$ and to emphasize its large values when defining
indices. In a simple way, this can be achieved by taking the $p$-th
power of $(f')^{-}$, $(f')^{+}$ and $f'$ for any $p\ge1$. This
argument leads us to the $L_p$-type counterpart of minimization problem
(\ref{def-00}) which can be solved along the same path as that used in
the proof of Theorem~\ref{theorem21}. It is remarkable that the
minimizing function does not depend on the choice of the metric and
remains equal to $(f')^{+}$. Consequently, for example, the following
$L_p$-index of lack of increase arises:
\[
\mathrm{LOI}_{p,y}(f)= \Biggl( \int_0^y
\bigl( \bigl(f' \bigr)^{-} \bigr)^p\,\mathrm{d}
\lambda \Biggr)^{1/p}.
\]
Other than $L_p$-norms can also be successfully explored, and this is
important in applications, where no $p$-th power may adequately
(de-)emphasize parts of $f'$. The phenomenon has prominently
manifested, for example, in Econometrics
(e.g., \cite{Cha1988,Zit2003}; and references therein). In such cases, more complexly
shaped functions are typically used, including those $H:[0,\infty) \to
[0,\infty)$ with $H(0)=0$ that give rise to the class of
Birnbaum--Orlicz (BO) spaces.

\section{Monotonicity comparisons and normalized indices}\label{sec4}

The index values in Example~\ref{example-21} suggest that on the noted
interval, sine is more increasing than cosine, because sine is closer
to the set $\mathcal{F}^{+}_y$ than cosine is. In general, given two
functions $g_1$ and $h_1$ on the interval $[0,y]$, we can say that the
function $g_1$ is more non-decreasing than $h_1$ on the interval
whenever $\mathrm{LOI}_y(g_1)\le\mathrm{LOI}_y(h_1)$. Likewise, we can
say that the function $g_2$ is more non-increasing than $h_2$ on the
interval $[0,y]$ whenever $\mathrm{LOD}_y(g_2)\le\mathrm{LOD}_y(h_2)$.
But an issue arises with these definitions because for a given pair of
functions $g$ and $h$, the property $\mathrm{LOI}_y(g)\le\mathrm
{LOI}_y(h)$ may not be equivalent to $\mathrm{LOD}_y(g)\ge\mathrm
{LOD}_y(h)$. Though it may look strange at first sight, this
non-reflexivity is natural because the total variations of the
functions $g$ and $h$ on the interval $[0,y]$ may not be equal, and in
such cases, comparing non-monotonicities of $g$ and $h$ is not meaningful.

If, however, the total variations of $g$ and $h$ are equal on the
interval $[0,y]$, then $\mathrm{LOI}_y(g)\le\mathrm{LOI}_y(h)$ if and
only if $\mathrm{LOD}_y(g)\ge\mathrm{LOD}_y(h)$. This suggests that in
order to achieve this `if and only if' property in general, we need to
normalize the indices, which gives rise to the following definitions
\[
\mathrm{LOI}^*_y(g)={\int_0^y (g')^{-}\,\mathrm{d}\lambda \over\int_0^y
\llvert g'\rrvert \,\mathrm{d}\lambda } \quad\textrm{and} \quad
\mathrm{LOD}^*_y(g)=\xch{{\int_0^y (g')^{+}\,\mathrm{d}\lambda \over\int_0^y
\llvert g'\rrvert \,\mathrm{d}\lambda },}{{\int_0^y (g')^{+}\,\mathrm{d}\lambda \over\int_0^y
\llvert g'\rrvert \,\mathrm{d}\lambda }}
\]
of the normalized indices of lack of increase and decrease,
respectively. Obviously,
%
\begin{equation}
\label{add-1} \mathrm{LOI}^*_y(g)+\mathrm{LOD}^*_y(g)=1,
\end{equation}
and thus $\mathrm{LOI}^*_y(g)\le\mathrm{LOI}^*_y(h)$ if and only if
$\mathrm{LOD}^*_y(g)\ge\mathrm{LOD}^*_y(h)$. Furthermore, the
normalized index of lack of monotonicity is $\mathrm{LOM}^*_y(g)
:=2\min \{ \mathrm{LOI}^*_y(g), \mathrm{LOD}^*_y(g) \}$, which
we rewrite as
%
\begin{equation}
\label{add-1.a} \mathrm{LOM}^*_y(g) = 1-{\llvert g(y)\rrvert  \over\int_0^y \llvert g'\rrvert \,\mathrm{d}\lambda}
\end{equation}
using the identity $\min\{u,v\}=(u+v-|u-v|)/2$ that holds for all real
numbers $u$ and $v$. In terms of the original function $g_0$ on the
interval $[a,b]$, we have
\[
\mathrm{LOI}^*_{[a,b]}(g_0)={\int_a^b (g'_0)^{-}\,\mathrm{d}\lambda
\over\int_a^b \llvert g'_0\rrvert \,\mathrm{d}\lambda }, \qquad
\mathrm{LOD}^*_{[a,b]}(g_0)={\int_a^b (g'_0)^{+}\,\mathrm{d}\lambda
\over\int_a^b \llvert g'_0\rrvert \,\mathrm{d}\lambda },
\]
and
\[
\mathrm{LOM}^*_{[a,b]}(g_0) = 1-{\llvert g_0(b)-g_0(a)\rrvert  \over\int_a^b \llvert g'_0\rrvert \,\mathrm{d}\lambda} .
\]
Obviously, $\mathrm{LOI}^*_{[a,b]}(g_0)=\mathrm{LOI}^*_y(g)$,
$\mathrm{LOD}^*_{[a,b]}(g_0)=\mathrm{LOD}^*_y(g)$, and
$\mathrm{LOM}^*_{[a,b]}(g_0)=\mathrm{LOM}^*_y(g)$. To illustrate the
normalized indices, we continue Example~\ref{example-21}.

\begin{example}\label{example-31}\rm
Recall that we are dealing with the functions $g_0(z)=\sin(z)$ and
$h_0(z)=\cos(z)$ on the interval $[-\pi/2,\pi]$. We transform them
into the functions $g(x)=1-\cos(x)$ and $h(x)=\sin(x)$ on the interval
$[0,3\pi/2 ]$. From equations (\ref{g1}) and (\ref{h2}), we see that
the total variations $\int_0^{3\pi/2} |g'|\mathrm{d}\lambda$ and $\int_0^{3\pi/2} |h'|\mathrm{d}\lambda$ are equal to $3$, and so we have
the equations:
\begingroup
\abovedisplayskip=7.5pt
\belowdisplayskip=7.5pt
\begin{align*}
\mathrm{LOI}^*_{[-\pi/2,\pi]}(g_0)=\mathrm{LOI}^*_{3\pi/2}(g)={1\over3}, & \quad \mathrm{LOI}^*_{[-\pi/2,\pi]}(h_0)=\mathrm{LOI}^*_{3\pi/2}(h)={2\over3},\\
\mathrm{LOD}^*_{[-\pi/2,\pi]}(g_0)=\mathrm{LOD}^*_{3\pi/2}(g)={2\over3}, & \quad \mathrm{LOD}^*_{[-\pi/2,\pi]}(h_0)=\mathrm{LOD}^*_{3\pi/2}(h)={1\over3},\\
\mathrm{LOM}^*_{[-\pi/2,\pi]}(g_0)=\mathrm{LOM}^*_{3\pi/2}(g)={2\over3}, & \quad \mathrm{LOM}^*_{[-\pi/2,\pi]}(h_0)=\mathrm{LOM}^*_{3\pi/2}(h)={2\over3}.
\end{align*}
\endgroup
The numerical procedure of Note~\ref{note-comp} can easily be employed
to calculate these normalized indices. This concludes Example~\ref{example-31}.
\end{example}

Next are properties of the normalized indices.
\begin{enumerate}
\item[\textrm{B}1)] The three indices $\mathrm{LOI}^*_y(g)$, $\mathrm
{LOD}^*_y(g)$, and $\mathrm{LOM}^*_y(g)$ are normalized, that is, take
values in the unit interval $[0,1]$, with the following special cases:
\begin{enumerate}
%
\item
$\mathrm{LOI}^*_y(g)=0$ if and only if $(g')^{-}\equiv0$, that is,
when $g$ is non-decreasing everywhere on $[0,y]$, and $\mathrm
{LOI}^*_y(g)=1$ if and only if $(g')^{+}\equiv0$, that is, when $g$ is
non-increasing everywhere on $[0,y]$.
\item
$\mathrm{LOD}^*_y(g)=0$ if and only if $(g')^{+}\equiv0$, that is,
when $g$ is non-increasing everywhere on $[0,y]$, and $\mathrm
{LOD}^*_y(g)=1$ if and only if $(g')^{-}\equiv0$, that is, when $g$ is
non-decreasing everywhere on $[0,y]$.
\item
$\mathrm{LOM}^*_y(g)=0$ if and only if $g$ is either non-decreasing
everywhere on $[0,y]$ or non-increasing everywhere on $[0,y]$, and
$\mathrm{LOM}^*_y(g)=1$ if and only if $\mathrm{LOI}^*_y(g)=\mathrm
{LOD}^*_y(g)$ (recall equation (\ref{add-1})).
\end{enumerate}

\item[\textrm{B}2)] The three indices $\mathrm{LOI}^*_y(g)$, $\mathrm
{LOD}^*_y(g)$, and $\mathrm{LOM}^*_y(g)$ are translation invariant,
that is, $\mathrm{LOI}^*_y(g+\alpha)=\mathrm{LOI}^*_y(g)$ for every
constant $\alpha\in\mathbf{R}$, and analogously for $\mathrm
{LOD}^*_y(g)$ and $\mathrm{LOM}^*_y(g)$.
\item[\textrm{B}3)]
\begin{enumerate}
%
\item
The three indices $\mathrm{LOI}^*_y(g)$, $\mathrm{LOD}^*_y(g)$, and
$\mathrm{LOM}^*_y(g)$ are positive-scale invariant, that is, $\mathrm
{LOI}^*_y(\beta g )=\mathrm{LOI}^*_y(g)$ for every positive constant
$\beta>0$, and analogously for $\mathrm{LOD}^*_y(g)$ and $\mathrm{LOM}^*_y(g)$.
\item
Moreover, $\mathrm{LOM}^*_y(g)$ is negative-scale invariant, that is,
$\mathrm{LOM}^*_y(\beta g )=\mathrm{LOM}^*_y(g)$ for every negative
constant $\beta< 0$, and thus, in general,\break $\mathrm{LOM}^*_y(\beta g
)=\mathrm{LOM}^*_y(g)$ for every real constant $\beta\ne0$ (recall
equation~(\ref{add-1.a})).
\end{enumerate}
\item[\textrm{B}4)]
$\mathrm{LOD}^*_y(- g )=\mathrm{LOI}^*_y(g)$, and thus $\mathrm
{LOD}^*_y(\beta g )=\mathrm{LOI}^*_y(g)$ for every negative constant
$\beta<0$.
\end{enumerate}

We next use the above indices to introduce three new orderings:
\begin{enumerate}
\item[C1)]
The function $g$ is more non-decreasing than $h$ on the interval
$ [0,y]$, denoted by $g \ge_{I,y} h$, if and only if $\mathrm
{LOI}^*_y(g)\le\mathrm{LOI}^*_y(h)$.
\item[C2)]
The function $g$ is more non-increasing than $h$ on the interval
$[0,y]$, denoted by $g \ge_{D,y} h$, if and only if $\mathrm
{LOD}^*_y(g)\le\mathrm{LOD}^*_y(h)$.
\item[C3)]
The function $g$ is more monotonic than $h$ on the interval $[0,y]$,
denoted by $g \ge_{M,y} h$, if and only if $\mathrm{LOM}^*_y(g)\le
\mathrm{LOM}^*_y(h)$.
\end{enumerate}

We see that on the interval $[0,y]$, the function $g$ is more
non-decreasing than $h$ if and only if the function $g$ is less
non-increasing than $h$. In other words, $g \ge_{I,y} h$ is equivalent
to $g \le_{D,y} h$, which we have achieved by introducing the
normalized indices.

\section{Stricter notion of comparison: a note}
\label{section-strict}

One of the fundamental notions of ordering random variables is that of
first-order stochastic dominance (e.g., \cite{Denetal2005,Lev2006,LiLi2013}). Similarly to this notation, our earlier
introduced orderings can be strengthened by first noting that the
integral $\int_0^y (g')^{-}\mathrm{d}\lambda$ is equal to $\int_0^{\infty} S^{-}_y(z\mid g) \mathrm{d}z$, where the function
\[
S^{-}_y(z\mid g)=\lambda \bigl\{x\in[0,y] ~:~
\bigl(g' \bigr)^{-}(x)> z \bigr\}
\]
counts the `time' that the function $(g')^{-}$ spends above the
threshold $z$ during the `time' period $[0,y]$. Likewise, we define the
`plus' version $S^{+}_y(z\mid g)$. This lead us to the following
definitions of ordering functions according to their monotonicity.
\begin{enumerate}
\item[D1)]
$g$ is more strictly non-decreasing than $h$ on the interval $[0,y]$,
denoted by $g \ge_{SI,y} h$, whenever
\[
{S^{-}_y(z\mid g)\over\int_0^y \llvert g'\rrvert \,\mathrm{d}\lambda } \le {S^{-}_y(z\mid h)\over\int_0^y \llvert h'\rrvert \,\mathrm{d}\lambda} \quad\textrm {for all}
\quad z> 0.
\]
\item[D2)]
$g$ is more strictly non-increasing than $h$ on the interval $[0,y]$,
denoted by $g \ge_{\mathit{SD},y} h$, whenever
\[
{S^{+}_y(z\mid g)\over\int_0^y \llvert g'\rrvert \,\mathrm{d}\lambda } \le {S^{+}_y(z\mid h) \over\int_0^y \llvert h'\rrvert \,\mathrm{d}\lambda} \quad\textrm{for all}
\quad z> 0.
\]
\end{enumerate}
Obviously, if $g \ge_{\mathit{SI},y} h $, then $ g \ge_{I,y} h$, and if $g \ge
_{\mathit{SD},y} h$, then $g \ge_{D,y} h$.

\section{Assessing lack of positivity in signed measures}
\label{section-hj}

We now take a path in the direction of general Measure Theory. Namely,
let $\varOmega$ be a set equipped with a sigma-algebra, and let $\mathcal
{M}$ denote the set of all (signed) measures $\nu$ defined on the
sigma-algebra. Furthermore, let $\mathcal{M}^{+} \subset\mathcal{M} $
be the subset of all positive measures. Given a signed-measure $\nu\in
\mathcal{M}$, we define its index of lack of positivity (LOP) by the equation
\[
\mathrm{LOP}(\nu) =\inf \bigl\{ \llVert\nu- \mu\rrVert~:~ \mu\in
\mathcal{M}^{+} \bigr\},
\]
where $\Vert\cdot\Vert$ denotes the total variation. Specifically,
with $(\varOmega^{-},\varOmega^{+})$ denoting a Hahn decomposition of $\varOmega
$, let $(\nu^{-},\nu^{+})$ be the Jordan decomposition of $\nu$. Note
that $\nu^{-}$ and $\nu^{+}$ are elements of $\mathcal{M}^{+}$. The
variation of $\nu$ is $|\nu|=\nu^{-}+\nu^{+}$, and its total variation
is $\Vert\nu\Vert=|\nu|(\varOmega)$. The following theorem provides an
actionable, and crucial for our considerations, reformulation of the
index $\mathrm{LOP}(\nu)$.

\begin{thm}\label{th-0}
The infimum in the definition of $\mathrm{LOP}(\nu)$ is attained on the
unique element of $\mathcal{M}^{+}$, which is the measure $\nu^{+}$,
and thus the $\mathrm{LOP}$ index can be written as
%
\begin{equation}
\label{th-0a} \mathrm{LOP}(\nu)=\bigl\llVert\nu^{-} \bigr\rrVert.
\end{equation}
\end{thm}

\begin{proof}
Since $\nu^{+}\in\mathcal{M}^{+}$, we have the bound $\mathrm{LOP}(\nu
)\le\Vert\nu- \nu^{+} \Vert$, which can be rewritten as $\mathrm
{LOP}(\nu)\le\Vert\nu^{-} \Vert$. To prove the opposite bound $\mathrm
{LOP}(\nu)\ge\Vert\nu^{-} \Vert$, we proceed as follows. For any $\mu
\in\mathcal{M}^{+}$, we have the bound
%
\begin{equation}
\llVert\nu- \mu\rrVert\ge\llvert \nu- \mu\rrvert \bigl(\varOmega^{-}
\bigr). \label{gen-00}
\end{equation}
Since $\nu=\nu^{+}-\nu^{-}$ and $\nu^{+}(A)=0$ for every $A\subset
\varOmega^{-}$, the right-hand side of bound (\ref{gen-00}) is equal to
$|\nu^{-}+\mu|(\varOmega^{-})$, which is not smaller than $\nu^{-}(\varOmega
^{-})$. The latter is, by definition, equal to $\Vert\nu^{-} \Vert$.
This establishes the bound $\mathrm{LOP}(\nu)\ge\Vert\nu^{-} \Vert$
and completes the proof of equation (\ref{th-0a}). We still need to
show that $\mu=\nu^{+}$ is the only measure $\mu\in\mathcal{M}^{+}$
such that the equation
%
\begin{equation}
\llVert\nu- \mu\rrVert=\bigl\llVert\nu^{-} \bigr\rrVert
\label{gen-01}
\end{equation}
holds. Note that $\Vert\nu^{-} \Vert=\nu^{-}(\varOmega^{-})$. Since
%
\begin{equation}
\llVert\nu- \mu\rrVert \ge \bigl\llvert \nu^{-}+ \mu \bigr\rrvert
\bigl( \varOmega^{-} \bigr)+ \bigl\llvert \nu^{+}- \mu \bigr
\rrvert \bigl(\varOmega^{+} \bigr), \label{gen-02}
\end{equation}
in order to have equation (\ref{gen-01}), the right-hand side of
inequality (\ref{gen-02}) must be equal to $\nu^{-}(\varOmega^{-})$. This
can happen only when $\mu(\varOmega^{-})=0$ and $|\nu^{+}- \mu|(\varOmega
^{+})=0$, with the former equation implying that the support of $\mu$
must be $\varOmega^{+}$, and the latter equation implying that $\mu$ must
be equal to $\nu^{+}$ on $\varOmega^{+}$. Hence, $\mu=\nu^{+}$. This
finishes the proof of Theorem~\ref{th-0}.
\end{proof}

Similarly to the LOP index, the index of lack of negativity (LON) of
$\nu\in\mathcal{M}$ is given by the equation
%
\begin{equation}
\label{th-0b} \mathrm{LON}(\nu)=\bigl\llVert\nu^{+} \bigr\rrVert,
\end{equation}
and the corresponding index of lack of sign (LOS) is
%
\begin{equation}
\mathrm{LOS}(\nu)=2 \min \bigl\{ \bigl\llVert\nu^{-} \bigr\rrVert,
\bigl\llVert \nu^{+} \bigr\rrVert \bigr\}. \label{th-0c}
\end{equation}
The reason for doubling the minimum on the right-hand side of equation
(\ref{th-0c}) is the same as in the more specialized cases discussed
earlier. Namely, due to the equation $\Vert\nu^{-} \Vert+ \Vert\nu
^{+} \Vert=\Vert\nu\Vert$, we have that $\min \{ \Vert\nu^{-}
\Vert,\Vert\nu^{+} \Vert \}$ does not exceed $\Vert\nu\Vert/2
$. Hence, for LOS to be always between $0$ and $\Vert\nu\Vert$, just
like LOP and LON are, we need to double the minimum. Finally, we
introduce the normalized indices
\[
\mathrm{LOP}^*(\nu)={\llVert\nu^{-} \rrVert\over\llVert\nu\rrVert}, \qquad \mathrm{LON}^*(\nu)=
{\llVert\nu^{+} \rrVert\over\llVert\nu\rrVert},
\]
and
\[
\mathrm{LOS}^*(\nu)=2 \min \bigl\{ \mathrm{LOP}^*(\nu),\mathrm {LON}^*(\nu)
 \bigr\},
\]
whose values are always in the unit interval $[0,1]$.

\section{Conclusion}
\label{section-conclude}

In this paper, we have introduced indices that, in a natural way,
quantify the lack of increase, decrease, and monotonicity of functions,
as well as the lack of positivity, negativity, and sign-constancy in
signed measures. In addition to being of theoretical interest, this
research topic also has practical implications, and for the latter
reason, we have also introduced a simple and convenient numerical
procedure for calculating the indices without resorting to frequently
unwieldy closed-form expressions. The indices satisfy a number of
natural properties, and they also facilitate the ranking of functions
according to their lack of monotonicity. Relevant applications in
Insurance, Finance, and Economics have been pointed out, and some of
them discussed in greater detail.

\section*{Acknowledgments}

We are indebted to Professor Yu. Mishura and two anonymous reviewers
for insightful comments and constructive criticism, which guided our
work on the revision. We gratefully acknowledge the grants ``From Data
to Integrated Risk Management and Smart Living: Mathematical Modelling,
Statistical Inference, and Decision Making'' awarded by the Natural
Sciences and Engineering Research Council of Canada to the second
author (RZ), and ``A New Method for Educational Assessment: Measuring
Association via LOC index'' awarded by the national research
organization Mathematics of Information Technology and Complex Systems,
Canada, in partnership with Hefei Gemei Culture and Education
Technology Co.~Ltd, Hefei, China, to Jiang Wu and RZ.


%
%
\end{document}